\pdfoutput=1
\RequirePackage{ifpdf}
\ifpdf % We~are running pdfTeX in pdf mode
\documentclass[pdftex]{sigma}%,draft
\else
\documentclass{sigma}
\fi

\usepackage[matrix,arrow,curve]{xy}
\usepackage{amsbsy}

\newtheorem{prop}{Proposition}
\newtheorem{conj}[prop]{Conjecture}

{\theoremstyle{definition}

\newtheorem{rem}[prop]{Remark} }

\begin{document}
\allowdisplaybreaks

\newcommand{\arXivNumber}{1810.09622}

\renewcommand{\PaperNumber}{115}

\FirstPageHeading

\ShortArticleName{The Full Symmetric Toda Flow and Intersections of Bruhat Cells}

\ArticleName{The Full Symmetric Toda Flow\\ and Intersections of Bruhat Cells}

\Author{Yuri B.~CHERNYAKOV~$^{\dag^1\dag^2\dag^3}$, Georgy I.~SHARYGIN~$^{\dag^1\dag^2\dag^4}$, Alexander S.~SORIN~$^{\dag^2\dag^5\dag^6}$\newline and Dmitry V.~TALALAEV~$^{\dag^1\dag^4\dag^7}$}

\AuthorNameForHeading{Yu.B.~Chernyakov, G.I.~Sharygin, A.S.~Sorin and D.V.~Talalaev}

\Address{$^{\dag^1}$~Institute for Theoretical and Experimental Physics, Bolshaya Cheremushkinskaya 25,\\
\hphantom{$^{\dag^1}$}~117218 Moscow, Russia}
\EmailDD{\href{mailto:chernyakov@itep.ru}{chernyakov@itep.ru}, \href{mailto:sharygin@itep.ru}{sharygin@itep.ru}, \href{mailto:talalaev@itep.ru}{talalaev@itep.ru}}

\Address{$^{\dag^2}$~Joint Institute for Nuclear Research, Bogoliubov Laboratory of Theoretical Physics,\\
\hphantom{$^{\dag^2}$}~141980 Dubna, Moscow region, Russia}
\EmailDD{\href{mailto:sorin@theor.jinr.ru}{sorin@theor.jinr.ru}}

\Address{$^{\dag^3}$~Institute for Information Transmission Problems, Bolshoy Karetny per.19, build.~1,\\
\hphantom{$^{\dag^3}$}~127994, Moscow, Russia}

\Address{$^{\dag^4}$~Lomonosov Moscow State University, Faculty of Mechanics and Mathematics,\\
\hphantom{$^{\dag^4}$}~GSP-1, 1 Leninskiye Gory, Main Building, 119991 Moscow, Russia}

\Address{$^{\dag^5}$~National Research Nuclear University MEPhI (Moscow Engineering Physics Institute),\\
\hphantom{$^{\dag^5}$}~Kashirskoye shosse 31, 115409 Moscow, Russia}

\Address{$^{\dag^6}$~Dubna State University, 141980 Dubna, Moscow region, Russia}

\Address{$^{\dag^7}$~Centre of integrable systems, P.G.~Demidov Yaroslavl State University,\\
\hphantom{$^{\dag^7}$}~150003, 14 Sovetskaya Str., Yaroslavl, Russia}

\ArticleDates{Received July 13, 2020, in final form November 02, 2020; Published online November 11, 2020}

\Abstract{In this short note we show that the Bruhat cells in real normal forms of semisimple Lie algebras enjoy the same property as their complex analogs: \textit{for any two elements~$w$,~$w'$ in the Weyl group $W(\mathfrak g)$, the corresponding real Bruhat cell $X_w$ intersects with the dual Bruhat cell $Y_{w'}$ iff $w\prec w'$ in the Bruhat order on $W(\mathfrak g)$}. Here $\mathfrak g$ is a normal real form of a~semisimple complex Lie algebra $\mathfrak g_\mathbb C$. Our reasoning is based on the properties of the Toda flows rather than on the analysis of the Weyl group action and geometric considerations.}


\Classification{22E15; 70H06}

\section{Introduction and the statement of result}
\subsection{General remarks}
Schubert cells in standard flag manifolds and their generalizations to arbitrary Lie groups -- Bruhat cells -- play an important role in algebraic geometry and representation theory. In the complex case their structure is well understood and has been extensively studied in the last 40 years.

It is natural to suppose that the intersection properties (described below) of real Bruhat cells are similar to the ones in the complex case. Although this has been implicitly assumed by many authors, the proof has not been accurately written down yet (or at least we couldn't find a~proper reference). The modest purpose of the present note is to fill this regrettable gap, at least for the normal real forms, and to show the close relation of the full symmetric Toda system with the geometry of Bruhat cells.

\subsection{Preliminaries on Bruhat decomposition}
The Bruhat cells (or Schubert cells) are open dense subsets in algebraic subvarieties in flag spaces (Schubert varieties), enumerated by the elements of the corresponding Weyl group. Namely, \textit{the Schubert cell $X_w^\mathbb C$, corresponding to $w\in W(\mathfrak g_\mathbb C)$}, is equal to the positive Borel subgroup ${\rm B}^+_\mathbb C$ orbit of the element $[w]\in {\rm Fl}(\mathfrak g_\mathbb C)=G_\mathbb C/{\rm B}^+_\mathbb C$. Here $[w]$ is the point in ${\rm Fl}(\mathfrak g_\mathbb C)$ represented by the element $\tilde w$ in the normalizer $N(\mathfrak h)$ of the Cartan subalgebra corresponding to $w\in W(\mathfrak g_\mathbb C)$. Similarly \textit{the dual Schubert cell $Y_w^{\mathbb C}$ corresponding to $w$} is the orbit ${\rm B}^-_\mathbb C\cdot [w]$ of the negative Borel subgroup.

It turns out that the properties of the Bruhat cells are closely related with the combinatorics of the corresponding Weyl group, in particular with the Bruhat order on it. Recall that the Bruhat order in a Coxeter group $W$ with the set $S=\{\sigma_1,\dots,\sigma_N\}$ of generators is the unique partial order on $W$ determined by the following rule: we say that $w$ precedes $w'$ if a minimal word $s(w')$ (in letters $\sigma_k$) representing~$w'$ has the form $\sigma_{i_1}\cdots \sigma_{i_k} \cdot s(w) \cdot \sigma_{j_1}\cdots\sigma_{j_m}$ (where $s(w)$ is a minimal word for~$w$). The lengths of these words are called \textit{lengths} of the elements and denoted $l(w)$, $l(w')$ respectively.

Now using complex geometry and algebraic groups theory one can show that $X^\mathbb C_w\subseteq \overline{X}^\mathbb C_{w'}$ iff $w\prec w'$ in Bruhat order (here $\overline X$ denotes the closure of $X$). Besides this, one can also show that \textit{$Y^\mathbb C_w\cap X^\mathbb C_{w'}\neq \varnothing$ iff $w'\prec w$ in Bruhat order} (see \cite{BrioLa,Ful}). Moreover, in the latter case the intersection is always transversal and the (complex) dimension of this variety is equal to the difference $l(w)-l(w')$. In particular, the complex dimensions of the cells are equal to~$l(w)$ and $\dim_\mathbb C {\rm Fl}(\mathfrak g_\mathbb C)-l(w')$ respectively (see the cited papers).

\subsection{Preliminaries on Cartan decomposition}
Let $\mathfrak g$ be a real Lie algebra whose complexification is equal to $\mathfrak g_\mathbb C$ (i.e., $\mathfrak g$ is a real form of $\mathfrak g_\mathbb C$). Let $G$ be the corresponding real Lie group. For any such \textit{real form} of $\mathfrak g_\mathbb C$ one has the following Cartan decomposition:
\[
\mathfrak g=\mathfrak k\oplus\mathfrak p,
\]
where $\mathfrak k$ is the Lie algebra of a maximal compact subgroup, and $\mathfrak p$ is the algebraic complement of $\mathfrak k$ on which the Killing form of $\mathfrak g$ is positive-definite.

Alternatively (and maybe more properly) one can define the Cartan decomposition with the help of a \textit{Cartan involution}. Namely let $\theta$ be a Lie algebra automorphism $\theta\colon \mathfrak g\to\mathfrak g$, such that $\theta^2=1$ and the quadratic form $B(X,\theta Y)$ where $B$ is the Killing form, is positive definite. Then
\[
\mathfrak k=\{X\in\mathfrak g\,|\, \theta(X)=-X\},\qquad \mathfrak p=\{Y\in\mathfrak p\,|\, \theta(Y)=Y\}.
\]

The Cartan decomposition of a real form $\mathfrak g$ is unique (up to an isomorphism) and one has
\[
[\mathfrak k,\mathfrak k]\subseteq\mathfrak k,\qquad [\mathfrak p,\mathfrak p]\subseteq\mathfrak k,\qquad [\mathfrak k,\mathfrak p]\subseteq\mathfrak p.
\]
Let $\mathfrak a$ be a maximal commutative Lie subalgebra in~$\mathfrak p$. Then \textit{the real form $\mathfrak g$ is called normal or non-split if $\mathfrak a$ is a maximal commutative subalgebra in whole $\mathfrak g$.}

In this case we shall denote this commutative subalgebra by $\mathfrak h$ and call it \textit{Cartan subalgebra}. The normal real form enjoys most of the properties of the complex algebra: one can determine the root system with respect to $\mathfrak h$ and decompose $\mathfrak g$ into root subspaces:
\[
\mathfrak g=\mathfrak h\oplus\sum_{\alpha\in\Phi}\mathfrak g_\alpha,
\]
where $\alpha\in\Phi\subset\mathfrak h^*$ are the roots and $\mathfrak g_\alpha$ are the root subspaces, i.e.,
\[
\mathfrak g_\alpha=\{X\in\mathfrak g\,|\, [H,X]=\alpha(H)X,\, H\in\mathfrak h\}.
\]
In the case of normal real form these subspaces are $1$-dimensional and fixing a basis in $\mathfrak h$ so that one can tell which roots are positive and which are negative, one has the following decompositions of $\mathfrak k$ and $\mathfrak p$ in the corresponding Cartan decomposition:
\begin{equation}\label{eq:2}
\mathfrak k=\sum_{\alpha\in\Phi_+}\mathbb R(e_\alpha+e_{-\alpha}),\qquad
\mathfrak p=\mathfrak h\oplus\sum_{\alpha\in\Phi_+}\mathbb R(e_\alpha-e_{-\alpha}),
\end{equation}
where $\Phi_+$ is the set of positive roots and $e_\alpha$, $e_{-\alpha}$ are basis vectors in $\mathfrak g_\alpha$, $\mathfrak g_{-\alpha}$ such that \mbox{$e_{-\alpha}=\theta(e_\alpha)$}. It is easy to see that in this case the Cartan decomposition of the corresponding complex Lie algebra $\mathfrak g_\mathbb C$ is given by the complexification of the corresponding elements of~$\mathfrak g$.

The Weyl group of $\mathfrak g$, which coincides with the Weyl group of $\mathfrak g_\mathbb C$, can now be defined as the quotient group $W(\mathfrak g)=N_K(\mathfrak h)/Z_K(\mathfrak h)$, where $N_K(\mathfrak h)$ and $Z_K(\mathfrak h)$ are the normalizer and the centralizer of $\mathfrak h$ in the maximal compact subgroup $K$ of $G$, corresponding to $\mathfrak k$. It follows from the maximality of $\mathfrak h$ that these groups are discrete and that $Z_K(\mathfrak h)$ is a normal subgroup in it.

\subsection{Main result}
One can define real Bruhat cells $X_w$ and dual real Bruhat cells $Y_w$ in ${\rm Fl}(G)=G/{\rm B}^+=K/Z_K(\mathfrak h)$ (where ${\rm B}^+$ is the real Borel subgroup) as the orbits of $[w]$ with respect to ${\rm B}^+$ and ${\rm B}^-$. In what follows below, we shall show that the following proposition is true.

\begin{prop}
The intersection $X_w\cap Y_{w'}$ is nonempty if and only if $w'\prec w$ in Bruhat order; moreover, if it is not empty, its (real) dimension is equal to the difference $l(w)-l(w')$.
\end{prop}

The explanation of this fact that we found in literature is based on the geometric description of Schubert cells in terms of matrix ranks, see, e.g.,~\cite{Ful}; it works well only in the case of $G = {\rm SL}_n(\mathbb R)$. The main purpose of this article is to extend this result to other groups by applying the connection of this theory with the full symmetric Toda flow. Namely, we are going to derive the result we need from the general properties of the Toda flow, and not vice-versa: in our previous work we used the properties of Schubert cells in the standard flag variety (i.e., the one associated with ${\rm SL}_n(\mathbb R)$) in order to show that the phase portrait of this dynamical system is completely determined by Bruhat order~\cite{CSS14}.

Also observe that in our other papers \cite{CSS19, CSS17} we used explicit computations to show that similar statements hold for all real groups of rank $2$ and for the non-split real form, ${\rm SO}(2,4)$. In the latter case the role of Weyl group is played by the group $W(\mathfrak g,\mathfrak k)=N_K(\mathfrak a)/Z_K(\mathfrak a)$ where~$N_K$,~$Z_K$ are the (non-discrete) normalizer and centralizer of $\mathfrak a$ in $K$. One can show that $W(\mathfrak g,\mathfrak k)$ is again a Coxeter group, thus we can introduce Bruhat order on it. Flag variety is replaced by $K/Z_K(\mathfrak a)$ and the (dual) Bruhat cells are defined as orbits of Borel subgroups in flag varieties similarly to the usual case. However, in this case the dimensions of the root spaces (with respect to $\mathfrak a$) are not necessarily equal to $1$. This is manifested in the dimensions of the intersections of the cells. Below we shall describe a possible generalization of our main result to this case.

\section{Full symmetric Toda flow and its properies}
\looseness=-1 In this section we recall few properties of the full symmetric Toda system on Lie algebras and flag spaces. For the connection between non-periodic Toda lattice and the geometric aspects of Lie groups see, for example~\cite{CaKo}. Namely, the full symmetric Toda system is the dynamical system on~$\mathfrak p$, defined as follows (we follow the exposition from the papers by de Mari, Pedroni~\cite{deMaPe} and Bloch, Gekhtman \cite{BloGe}): it is the Hamiltonian system with respect to a Poisson structure, pulled to $\mathfrak p$ from $(\mathfrak b^+)^*$ via the Killing form, and with the Hamiltonian $H(L)=\operatorname{Tr}_\mathfrak g\big({\rm ad}^2_\mathfrak g (L)\big)$, $L\in\mathfrak p$.

However, for our purposes it is better to consider only the corresponding Hamiltonian vector field $\tau$ on $\mathfrak p$; it can be written down in the following commutator form:
\[
\tau(L)=[L,M(L)].
\]
Here for
\[
L=X+\sum_{\alpha\in\Phi_+}a_\alpha(e_\alpha+e_{-\alpha})\in\mathfrak p,\qquad X\in\mathfrak h,\qquad a_\alpha\in\mathbb R
\]
(see \eqref{eq:2}) we put
 \[
 M(L)=\sum_{\alpha\in\Phi_+}a_\alpha(e_\alpha-e_{-\alpha})\in\mathfrak k.
 \]
Thus the Toda system can be written in the following Lax form:
\begin{equation}\label{eq:Toda1}
\dot L=[L,M(L)].
\end{equation}
Below we shall refer to the map $M\colon \mathfrak p\to\mathfrak k$ as to the \textit{Toda projector}, and the system~\eqref{eq:Toda1} as the \textit{full symmetric Toda system on $\mathfrak g$}, or by a slight abuse of terminology, just \textit{Toda system on $\mathfrak g$} since it is clear that it depends only on the real form $\mathfrak g$ (conventionally, the term Toda system is used to denote the dynamical system determined by~\eqref{eq:Toda1} on the subspace in~$\mathfrak p$ spanned by $\mathfrak h$ and the simple root spaces).

It turns out that the dynamics of the Toda system \eqref{eq:Toda1} is uniquely determined by the following construction: \textit{it is known that almost all $($i.e., up to a measure $0$ subset$)$ orbits of $K$ on $\mathfrak p$ pass through a unique $($up to the action of $N \in\mathfrak h)$ element of $\mathfrak h$}; on the other hand, the vector field $\tau$ is evidently tangent to these orbits (because it is determined by the adjoint action of $\mathfrak k$). So we choose a generic element $\Lambda\in\mathfrak h$ such that its centralizer coincides with the centralizer of $\mathfrak h$ and consider the vector field $T$ at a point $\Psi\in K$, given by
 \[
 T^\Lambda(\Psi)=-M(\operatorname{Ad}_\Psi(\Lambda))\cdot \Psi,
 \]
i.e., it is the right translation by $\Psi$ of the element $M(\operatorname{Ad}_\Psi(\Lambda))\in\mathfrak k$. Then this vector field gives a~dynamical system $\Psi=\Psi(t)$ on~$K$, such that the solution of Toda system on $\mathfrak p$ through a~generic point $L=L(0)$ is given by the formula
 \[
 L(t)=\operatorname{Ad}_{\Psi(t)}(\Lambda),
 \]
where $\Lambda\in\mathfrak h$ is the eigenmatrix of $L$ (the one fixed above) and $\Psi(t)$ is the trajectory of the field~$T$ on~$K$, such that $\operatorname{Ad}_{\Psi(0)}(\Lambda)=L$. It would probably be more accurate to consider instead of the field~$T$ on~$K$ an analogous field $T'$ on the flag space. But in that case the formula for~$L(t)$ would involve a choice of section $F(G)\to K$.

\subsection{Invariant surfaces of the Toda field}
It is our purpose to study the behaviour of the vector field $\tau$ on $\mathfrak p$ by studying the equivalent fields $T^\Lambda$ on $K$. It has been studied by many authors in the last 25 years so here we shall just give references for the well-known facts and give brief proofs for those which seem to be less known.

One can show (see \cite{BloGe} and \cite{BBR}) that the vector field $T^\Lambda$ on $K$ is gradient. The corresponding potential function $F=F_\Lambda\colon K\to\mathbb R$ is given by
\[
F_\Lambda(\Psi)=\operatorname{Tr}_\mathfrak g(\operatorname{ad}_\mathfrak g(\operatorname{Ad}_\Psi(\Lambda))\operatorname{ad}_\mathfrak g(N)),
\]
for a suitable element $N\in\mathfrak h$. It turns out that $F_\Lambda$ is a Morse function. The construction also involves a~proper choice of the Riemann structure $\langle\,,\,\rangle$ on $K$. We shall not use the exact form of the metric and of the element $N$ here. It is enough to say that the scalar product in question is left-invariant and that in the basis $e_\alpha-e_{-\alpha}$ of the Lie algebra $\mathfrak k$ it coincides with the Killing form up to scalar factors $j_\alpha$ and $N$ is uniquely defined up to a central element of~$\mathfrak g$, look into~\cite{CSS14, deMaPe} for details.

From the very definition of the field $T^\Lambda$ it follows that for a generic $\Lambda$ the only singular points of this field coincide with the group $N_K(\mathfrak h)$ or with the set of elements of $W(\mathfrak g)$ on the level of the flag space ${\rm Fl}(G)$. Further, one can show that there always exists a vast family of invariant (with respect to $T^\Lambda$) curves in $K$, i.e., of $1$-dimensional invariant submanifolds of the Toda flow. In fact, it is possible to give an explicit description of such curves. To this end we let $w\in N_K(\mathfrak h)\subset K$ be a point and let us fix a root $\alpha\in\Phi_+$. Consider the smooth path $\gamma_w=\gamma_{w,\alpha}$ through the point $w$ on $K$ given by
 \[
 \gamma_w(t)=\exp{(t(e_\alpha-e_{-\alpha}))}\cdot w.
 \]
 \begin{prop} \label{prop:1dim}
The path $\gamma_w$ is tangent to the vector field $T$ on $K$.
 \label{parallel}
 \end{prop}
 \begin{proof} Let us consider the path
 \[
 \gamma_{w,\alpha}(t)=\exp{(t(e_\alpha-e_{-\alpha}))}\cdot w,
 \]
 and the corresponding path in the Lie algebra
 \[
 L(t)=\operatorname{Ad}_{\gamma_{w,\alpha}(t)}(\Lambda)=\operatorname{Ad}_{\exp{(t(e_\alpha-e_{-\alpha}))}}(\Lambda_{w}),
 \]
 where the element $\Lambda_w\in\mathfrak h$ is given by
 \[
 \Lambda_w=\operatorname{Ad}_w(\Lambda).
 \]
Since the action of $\mathfrak h$ on $e_{\pm\alpha}$ is diagonal, we can prove by induction that $L(t)$ is an element in the subspace in $\mathfrak p$ generated by $\mathfrak h$ and ${e}_\alpha + e_{-\alpha}$. Hence $M(L(t))$ is proportional to $e_\alpha-e_{-\alpha}.$ Let us introduce a scalar function $k(t)$ by
 \[
 M(L(t))=k(t)(e_\alpha-e_{-\alpha}).
 \]
 Then the Toda field $T$ at the point $\gamma_w(t)$ is equal to $k(t)(e_\alpha-e_{-\alpha})\cdot \gamma_w(t)$ and is proportional to the tangent field to the trajectory $\gamma_{w,\alpha}(t)$ with coefficient $k(t)$.
 \end{proof}

In addition to these invariant curves we shall need the following result, which one can find, for instance in \cite{deMaPe, Fay}: \textit{Bruhat cells are invariant manifolds of this system}. In fact, one can make this statement more precise, bringing it to the following form.

For every $\Lambda\in\mathfrak h$ we consider the ordering of the elements of $W(\mathfrak g)$ by the values of $F_\Lambda$; let $w\in W(\mathfrak g)$ be the minimal point. Then for any $w'$ we consider the shifted Bruhat cell~$X_{w'}^w$ and the dual shifted Bruhat cell $Y_{w'}^w$, i.e., the orbits $({\rm B}^+)^{w}\cdot[w']$ and $({\rm B}^-)^{w}\cdot[w']$, where $({\rm B}^\pm)^w=w\cdot {\rm B}^\pm\cdot w^{-1}$ and $[w']$ is the point in ${\rm Fl}(G)$ corresponding to $w'$. Then \textit{the shifted and dual shifted Bruhat cells are invariant submanifolds of~$T^\Lambda$. Moreover Bruhat cells are unstable and dual Bruhat cells are stable submanifolds of this system.} See the cited papers for proof.

Before we proceed, let us make the following observation: \textit{Bruhat cells, dual Bruhat cells and their intersections are invariant surfaces of vector fields $T^\Lambda$ for all $\Lambda\in\mathfrak h$}. This follows directly from the previously cited facts: just observe that if two vector fields are tangent to a surface, then all their linear combinations also are and notice that the fields $T^\Lambda$ are linear in~$\Lambda$.

\section{Proof of the main result}
We begin with the following simple observation: \textit{if the intersection $X_w\cap Y_{w'}$ is nonempty, then the elements $w$ and $w'$ are comparable in Bruhat order so that $w'\prec w$.} In fact real Bruhat cells lie inside complex ones and in the complex case intersection of cells is equivalent to the comparison $w'\prec w$. Also observe, that the real dimension of the intersections of the real Bruhat cells with dual real Bruhat cells does not exceed $l(w)-l(w')$ since this is the complex dimension of the intersection of complex cells.

Further, suppose that $\Lambda\in\mathfrak h$ is such that the minimal point of the Morse function $F_{\Lambda}$ is equal to the unit element $e\in W(\mathfrak g)$ (this is possible since $F_\Lambda(\Psi\cdot w)=F_{\operatorname{Ad}_w(\Lambda)}(\Psi)$. So if we choose $\Lambda'=\operatorname{Ad}_w(\Lambda)$, where $w$ is the current minimal element, then $e$ will be minimum of $F_{\Lambda'}$). Then it follows from the properties of Bruhat cells and Toda system, listed at the end of the previous section, that if there exists a trajectory $k(t)$ of $T^\Lambda$ which tends to $w$ when $t\to+\infty$ and to $w'$ when $t\to-\infty$ then $w'\prec w$. Indeed, in this case the corresponding real Bruhat and dual real Bruhat cells intersect and hence complex cells intersect too.

Consider now the paths $\gamma_{w,\alpha}$ corresponding to the roots $\alpha_1,\dots,\alpha_N$. It is known from the classical Lie groups theory that there exists $t_0\in\mathbb R$ such that $\exp(t_0(e_\alpha-e_{-\alpha}))=\sigma_\alpha$, where~$\sigma_\alpha$ is the reflection of $\mathfrak h$, corresponding to $\alpha$. To verify this one can choose the eigenfunction basis. Thus, the paths $\gamma_{w,\alpha_i}$ connect $w$ and $w'=\sigma_\alpha\cdot w$. Since these paths coincide (up to reparametrization) with the trajectories of the Toda system, which is gradient system with stable points equal to~$W(\mathfrak g)$, we conclude from the definition of Bruhat order that there exist trajectories of Toda system connecting any two elements of~$W(\mathfrak g)$ which differ by only one reflection.

Let now $w'\in W(\mathfrak g)$ be any element equal to the product of exactly two reflections. Let us prove that there exists a trajectory of the Toda system connecting the unit $e\in W(\mathfrak g)$ and $w'$. But it follows from the properties of Toda system, listed in the end of previous section, that the dimension of unstable manifold of $w'$ is equal to the dimension of the corresponding Bruhat cell, i.e., equals $2$. On the other hand, there can be no trajectories entering $w'$ from the elements whose length is greater or equal to $2$ (since these points do not verify the condition $w\prec w'$). Similarly, the dimension of the space, spanned by the trajectories, connecting $w'$ with elements of length $1$, can not exceed $1$ (it follows from the first observation in this section). Thus, this space should consist of the trajectories coming from $e$ (up to a measure zero subset, spanned by the trajectories from ``intermediate'' points). It means that the set of such trajectories is nonempty, i.e., the corresponding Bruhat cell and the dual Bruhat cell intersect. It also follows that the dimension of this intersection is $2$.

Next, if $w\prec w'$ and $l(w')-l(w)=2$ we choose $\Lambda'$ so that $w$ is minimal (this is done as before by passing to the function $F_\Lambda(\Psi\cdot w)=F_{\operatorname{Ad}_w(\Lambda)}(\Psi)$ for suitable~$w$. One can also say that this amounts to the change of the set of positive roots or to the action of Weyl group on the Weyl chambers of $\mathfrak h$). In this new setting $w$ will play the role of unit in Weyl group and $w'$ will correspond to an element of length~2. Then as before, we conclude that there exist trajectories of $T^{\Lambda'}$ from $w$ to $w'$ and the space of these trajectories is $2$-dimensional. Now by the remark we made in the end of the previous section this surface is also invariant for $T^\Lambda$, thus, since $w$ is local minimum and $w'$ is maximum of $F_\Lambda$ on this surface (in fact $F_{\Lambda'}=w^*(F_\Lambda)$ and the Riemannian structure used to determine the Toda flow is (right) invariant) we conclude that there exists a~$2$-dimensional space of trajectories of $T^\Lambda$ from $w$ to $w'$. And hence the corresponding Bruhat and dual Bruhat cell intersect along a $2$-dimensional variety.

Finally, we proceed by induction: assuming that we know the existence of intersections of all cells with $l(w)-l(w')\le k$, we consider first $w'=e$, $l(w)=k+1$ and conclude (by dimension counting) that the space of trajectories from~$e$ to~$w$ is nonempty and has dimension $k+1$. Then, using the shift we reduce the general case to this one.

Observe that in addition to the main proposition we have also proved that \textit{two stable points $w$ and $w'$ of Toda system on a normal real form are connected by a trajectory iff they are comparable in Bruhat order.}

\begin{rem}It seems that the same intersection result can be obtained from the standard reasoning, based on the combinatoric considerations and on the use of Bott--Samelson varieties and their generalizations similarly to the Deodhar's construction for algebraically-closed fields, see~\cite{Deodh}.
\end{rem}

\subsection{Non-split case and further questions}
In non-split case we have the problem of identifying the real Bruhat cells with the real part of complex cells. Thus the very first stage of our inductive process fails. In fact, the Weyl group~$W(\mathfrak g,\mathfrak k)$ in this case coincides with the Weyl group of the smaller root system, namely, of the so-called Satake projection in which $\mathfrak a$ plays the role of Cartan algebra. Besides this, some of the root spaces~$\mathfrak g_\alpha$ can be multidimensional.

On the other hand, most part of the statements from the previous section remain true: the points $w\in W(\mathfrak g,\mathfrak k)$ are invariant points of Toda system, their Bruhat cells and dual Bruhat cells (or their shifted versions) are unstable and stable manifolds of the Toda system (see~\cite{Fay}) and the paths $\gamma_{\alpha,w}$ are invariant with respect to Toda flow.

Using these facts, one can make the following conjecture, concerning the structure of trajectories of Toda flow on non-split real forms, and the intersections of the Bruhat cells:
\begin{conj}
There exist trajectories connecting $w$ and $w'$ in $W(\mathfrak g,\mathfrak k)$ iff the corresponding elements are comparable in Bruhat order, in which case the space of such trajectories coincides with the intersection of Bruhat and dual Bruhat cell, and its dimension is equal to the sum:
\[
\sum_{i=1}^k\dim\mathfrak g_{\alpha_i},
\]
where $w=\sigma_{\alpha_{i_k}}\cdots\sigma_{\alpha_{i_1}}\cdot w'$. In particular, this sum does not depend on the choice of the reflections $\sigma_{\alpha_i}$ in this formula.
\end{conj}
Another interesting question is related to the fact that in Deodhar's work~\cite{Deodh} one obtains further decomposition of the intersections of Bruhat cells into smaller algebraic varieties. One can ask if these varieties can also be somehow characterized in terms of the Toda system.

\subsection*{Acknowledgments}
The authors thank Sergey Loktev, Vladimir Rubtsov, Evgeniy Smirnov for useful discussions. The authors would like to express their gratitude to the anonymous referee, whose remarks have helped us to improve the presentation of this paper.

The work of Yu.B.~Chernyakov was partly supported by the grant RFBR-18-02-01081. The work of G.I.~Sharygin was partly supported by the grant RFBR-18-01-00398. The work of D.V.~Talalaev was partly supported by the grant Leader(math) 20-7-1-21-1 of the foundation for the advancement of theoretical physics and mathematics ``BASIS'' and within the framework of a development program for the Regional Scientific and Educational Mathematical Center of the Yaroslavl State University with financial support from the Ministry of Science and Higher Education of the Russian Federation (Agreement No.~075-02-2020-1514/1 additional to the agreement on provision of subsidies from the federal budget No.~075-02-2020-1514).

\pdfbookmark[1]{References}{ref}
\LastPageEnding

\end{document}